\include {babel}
\documentclass[12pt]{amsart}
\usepackage[mathscr]{eucal}
\usepackage{amssymb,amsmath}
\usepackage{amsfonts}
\usepackage{a4}
\usepackage[T1]{fontenc}
\usepackage[all]{xy}
\usepackage{graphicx}
\begin{document}
\author {Miros\l aw Sobolewski}
 \address {University of Warsaw,\newline Faculty of Mathematics,Informatics and Mechanics, \newline Instytut Matematyki, \newline Banacha 2, Warszawa  02-097,Poland}
 \email {msobol@mimuw.edu.pl}
\begin {abstract}Klee introduced   \emph{the proximate fixed point property} for compacta   which is stronger than fixed point property.
We consider relations between proximate fixed point property of spaces being result of application of different operations to continua. As  an application we show this property for products, cones, suspensions and joins of span 0 continua.\end{abstract}
\title{Proximate fixed point property and operations}
\maketitle
 \maketitle
\newcommand*{\e}{\k e}
\newcommand*{\as}{\k a}
\newtheorem{tw}{Theorem}
\newtheorem{pr}{Proposition}
\newtheorem{st}{Claim}
\newtheorem{df}{Definition}
\newtheorem{lem}{Lemma}
\newcommand*{\NWD}{\operatorname*{NWD}}
\newcommand*{\diam}{\operatorname*{diam}}
\newcommand*{\dist}{\operatorname*{dist}}
\newcommand* \wn{\operatorname*{int}}
\newcommand* \cl{\operatorname*{cl}}
\newcommand* \bd{\operatorname*{bd}}
\section{Introductory  lemmas}
All spaces are considered to be subspaces of an AR compactum $Q$. The closed $\varepsilon$-neighborhood of $X$, ie., $\{x\in Q:\rho (x,X)\leq \varepsilon\}$ of a set $X$ will be denoted by $X^\varepsilon$, it is assumed to be an ANR. Notation $x\stackrel{\varepsilon} {=}y$ means that $\rho(x,y)\leq \varepsilon$
Let $X, Y$ be compacta, and $\varepsilon>0$ . We say that a  function $f:X\to Y$ is    $\varepsilon$- \emph{continuous}  if there exist a number $\delta>0$ such that if $\rho (x,x')<\delta$ then $\rho(f(x),f(x') <\varepsilon$. Victor Klee in \cite{Klee}  introduced the following
\begin{df} A compactum $X$ has the \emph{proximate fixed point property } (shortly pfpp) if for every $\varepsilon>0$ there exists a number $\delta>0$ such that
if $f:X\to X$ is $\delta$-continuous then there exists a point $x\in X$ such that $x\stackrel {\varepsilon}{=} f(x)$\end{df}. The pfpp is a topological property and if a compactum has the pfpp then it has the fixed point property (Klee).
\begin{figure} \begin{centering}
\includegraphics[angle=0, width=4in ]{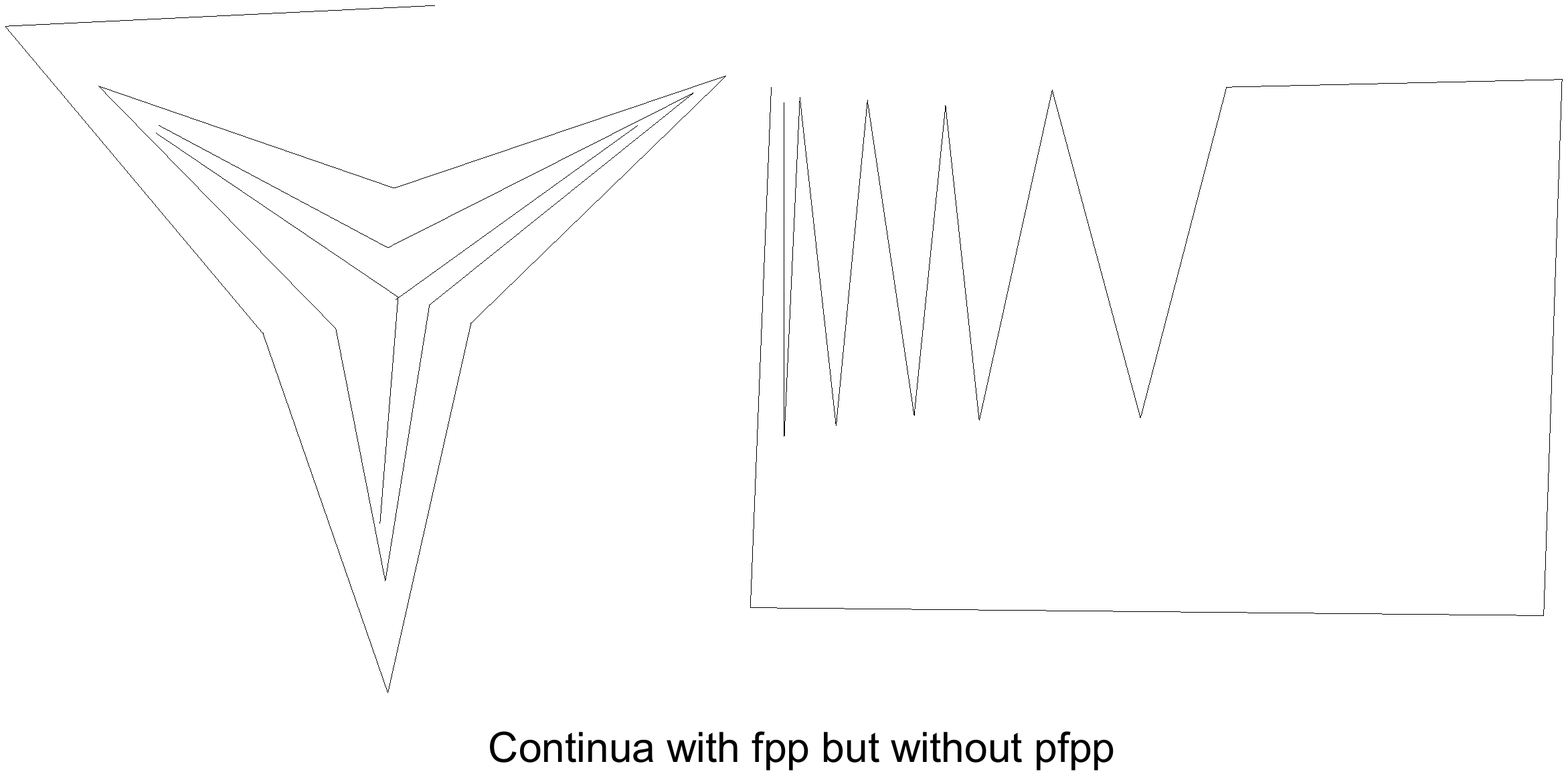}
\caption{.}\label{przykl}\end{centering}
\end{figure}
In  fig.1 are presented examples of continua with the fixed point property but without the pfpp -- a warsaw circle nad a spiral winding up a triod. If an ANR has the fixed point property then it has the pfpp as well (Klee).
Jose Sanjurjo showed in \cite{Sanj} the following \begin{tw}[Sanjurjo] A compactum $X$ has the proximate fixed point property $\Leftrightarrow$ for every $\varepsilon>0$ there exists $\delta>0$ such that if $f$ is a mapping from $X$ to $X^\delta$ then $x\stackrel{\varepsilon}=f(x)$ for a point $x\in X$.  \end{tw}
We will use this rather than original definition of Klee.
There are some simple but useful counterparts of well known fixed point theorems:
Let us recall, that a mapping $f:X\to Y$ is\emph{universal} if for every $g:X\to Y$ there exists $x\in X$ such that $f(x)=g(x)$ (Holsztyñski).
\begin{pr}\label{anraproks} Let $X$ be a continuum. If for every $\varepsilon>0$ there exists a universal $\varepsilon$-mapping $f:X\to Y$, where  $Y_\varepsilon$ ia an ANR  then $X$ has the pfpp.\end{pr}\begin{proof} Let $f$ be a universal  $\varepsilon/2$-mapping of $X$ onto $Y$ -- an ANR compactum $Y$. For some $\delta>0$ we can extend it to an $\varepsilon/2$-mapping $g:X^\delta\to Y$. Now let $h:X\to X^\delta$. Because $f$ is universal there exists a point $x\in X$ such that $f(x)=gh(x)$ and hence $f(x)\stackrel{\varepsilon}{=} x$.\end{proof}
A consequence of Proposition (\ref{anraproks}) is:
\begin{pr} Each ANR having fpp has pfpp as well.\end{pr}

\begin{pr}[Klee] Every retract $Y$ of a continuuum $X$ with the pfpp has the pfpp as well.\end{pr}

We call a mapping $f:X\to Y$ an $\varepsilon$-\emph{shift} if $x\stackrel{\varepsilon}{=}f(x)$ for every $x\in X$
\begin{pr}\label{slaby} If for every number $\varepsilon>0$ there exists an $\varepsilon$-shift of a continuum $X$ into its subcontinuum $X_\varepsilon$ with pfpp then $X$ has pfpp.\end{pr}
For us however, there will be useful a strengthening:
\begin{pr}\label{mocny} Let $X$ be a continuum .If for every number $\varepsilon>0$ there exists a subcontinuum $X_\varepsilon$ with pfpp such that for every number $\xi>0$ there exists an $\varepsilon$-shift $f_\varepsilon^\xi : X\to X_\varepsilon^\xi$   then $X$ has pfpp.\end{pr}\begin{proof}Let $\varepsilon>0$ and let $\lambda=\varepsilon/3$. Choose $\eta>0$ such that for every $\phi: X_\lambda\to X_\lambda^\eta$ there exists a point $x_0\in X_\lambda$ such that $x_0\stackrel{\lambda}{=}\phi(x_0)$. Consider a $\lambda$-shift $f:X\to X_\lambda^{\eta/2}$. It can be extended to a $2\lambda$-shift $\tilde f:X^\delta \to X_\lambda ^\eta$ for a number $\delta>0$. Now, let $g:X\to X^\delta$. The composition $\tilde f\circ g$ maps $X_\lambda$ into $X_\lambda^\eta$, hence we have a point $x_0\in X_\lambda\subset X$ such that $x_0\stackrel{\lambda}{=}\tilde f\circ g(x_0)\stackrel {2\lambda}{=}g(x_0)$. Thus $g(x_0)\stackrel{\varepsilon}{=}x_0$.\end{proof}
Sometimes is useful an outer approximation. One can easily show the following:
\begin{pr} Let $X$ be a continuum. If for every number $\varepsilon>0$ there exists an $\varepsilon$-shift $f:X_\varepsilon\to X$, where $X_\varepsilon$ is a continuum with pfpp such that $X\subset X_\varepsilon$ then the continuum $X$ has pfpp.\end{pr}
A countable product of compacta has small projections onto finite subproducts, hence from Proposition \ref{slaby} we have:
\begin{pr}\label{infprod}If for a family of continua $X_i$, $i=1,2,\dots$ each cartesian product of finite subfamily has pfpp then the product $\prod_{i=1}^\infty X_i$ has pfpp as well.\end{pr}
Let us recall that a continuum $X$ has \emph{trivial shape} if inclusion mapping $\iota :X\to X^\varepsilon$ is homotopic to a mapping into a point for every number $\varepsilon>0$ \cite{Bo}).
\begin{pr} \label{cstrsh}Denote $\mathbb I=[-1,1]$. Let    $X$ be a continuum of trivial shape. If the cylinder $X\times \mathbb I$ has pfpp then the cone $C(X)$ and the suspension $S(X)$ have pfpp as well.\end{pr}\begin{proof}The cone $C(X)$ is the quotient space $X\times [0,1]/X\times \{1\}$. The continuum $X$ is
of trivial shape, hence for every number $\xi>0$ the embedding $\iota:X\times [-1,1]\to X^\xi$ extends to a mapping $\tilde \iota:C(X)\to X^\xi$. Replacing the whole cone $C(X)$ with its $\varepsilon/2$-top $X\times[1-\varepsilon/2,1]/X\times \{1\}$ we obtain an $\varepsilon$-shift of the $\varepsilon/2$-top into arbitrarily close neighborhood, being identity on the bottom of the $\varepsilon/2$-top. We can extend it on the whole $C(X)$ defining the extension
to be identity on the truncated cone corresponding to $X\times [0,1-\varepsilon]$. The truncated cone is homeomorphic to the cylinder $X\times\mathbb I$. Hence we have for every $\varepsilon>0$ an $\varepsilon$-shift into arbitrarily close neighborhood of a continuum with pfpp in $C(X)$. That means by Proposition \ref{mocny} that $C(X)$ has pfpp. A proof for $S(X)$ goes quite alike.\end{proof}
Recall that \emph{join} of spaces $X$ and $Y$, denoted  $A\star B$, is the quotient space of the product $X\times Y\times [0,1]$ divided by relation generated by
    $(x, y_1, 0) \sim (x, y_2, 0) \quad\mbox{for all } x \in X\mbox{ and } y_1,y_2 \in Y,
    (x_1, y, 1) \sim (x_2, y, 1) \quad\mbox{for all } x_1,x_2 \in X \mbox{ and }y \in Y. $
\begin{pr}\label{jointrsh}Let $X$  be a continuum of trivial shape. If the product of continua $X\times Y\times \mathbb I$ has pfpp then the product $C(X)\times Y$ has pfpp.
Moreover if $Y$ is of trivial shape too then  the join  $X\star Y$ has pfpp as well.\end{pr}\begin{proof} in the case of $C(X)\times Y$ we can consider it as a quotient space of the product $X\times I\times Y$ and using small shifts of $C(X)$ into arbitrarily close neighborhood of a truncated cone we obtain small shifts of $C(X)$ (see proof of previous theorem) into arbitrarily close neighborhood of $X\times [0,1-\varepsilon]\times Y$ which
is homeomorphic to $X\times Y\times \mathbb I$, hence has pfpp. In the case of the join we can consider similar small shifts into arbitrarily close neighborhoods of
parts of the join corresponding to $X\times Y\times [\varepsilon,1-\varepsilon]$. In both cases applying Proposition\ref{mocny} we obtain pfpp for $C(X)\times Y$ and $X\star Y$.\end{proof}

\section{The proximate fixed point property and  span 0 continua.}
Lelek has introduced in \cite{Lel} following notions. A continuum $X$ is said to be \emph{span 0 continuum}, shortly $\sigma(X)=0$ if for every continuum $K$ and mappings $f,g;K\to X$ such that $f(K)=g(K)$ there exists a point $k\in K$ such that $f(k)=g(k)$. A continuum  $X$ is of \emph{surjective semispan} 0, shortly $\sigma^*_0(X)=0$,  if for every two mappings $f,g;K\to X$ , such that $f$ is surjective,of a continuum $K$ there exists $k\in K$ such that $f(k)=g(k)$. The continuum $X$ has 0 semispan if its every subcontinuum has 0 surjective semispan. A continuum has 0 span if and only if it has 0 semispan (\cite{Dv}). That means that a 0 span continuum has 0 surjective semispan It is easy to show the following property of a continuum $X$ of surjective 0 semispan:
\begin{lem}\label{span} Let $X$ be a continuum and let $\sigma^*_0(X)=0$. Then for every $\varepsilon>0$ there exists $\delta>0$ such that

*) If $f,g:K\to X^\delta $ are mappings of a continuum  $K$ such that the Hausdorff distance of $f(K)$ to $X$ is less than $\delta$ then there exists a point $k\in K$ such that $f(k)\stackrel{\varepsilon}{=}g(k)$.\end{lem}
\begin{proof}Suppose not, ie., there exists a number $\varepsilon>0$ such that for every $\delta >0$ there exist  mappings $f,g:K\to X^\delta $ with Hausdorff distance of $f(K)$ from $X$ less than $\delta$. Consider in $Q\times Q$ the set $L_\delta=\{(f(k),g(k)): k\in K\}$. We can choose a convergent sequence of $L_\delta$ with $\delta\to 0$, and denote $\tilde L$ the limit continuum. Then we have that the projection of $\tilde L$ onto first coordinate equals $X$  and the projection on the  second coordinate contains in $X$ and for every $\zeta \in L_\delta$ the distance between $\pi_1(\zeta)$ and $\pi_2(\zeta)$ is greater than $\varepsilon$. This violates $\sigma^*_0(X)=0$\end{proof}
 We need a
 \begin{lem} \label{essen} Let $D$ denote the unit ball in a space $\mathbb R^n$ (in the case of $\mathbb R$ we have $D=\mathbb I$). Let for a continuum $X$ the cartesian product $X\times D$ has pfpp. Then for every number $\varepsilon>0$ there exists  a number $\delta>0$ such that

 **) if  $f:X\times D\to X^\delta $ then the restriction of the projection  $\pi_D:X\times D\to D$ to the set $\Gamma (f,\varepsilon)=\{(x,y)\in X\times D: F(x,y)\stackrel {\varepsilon}{=} x\}$ is essential.\end{lem}
 \begin{proof} Suppose not, i.e., for each $\delta>0$ there exists $f:X\times D\to X^\delta$ such that the restriction $\pi_D|\Gamma (f,\varepsilon)$ is inessential. Hence there exists a mapping $\Phi:\Gamma(f,\varepsilon)\to\partial D$ such that $\Phi(\zeta)=\pi_D(\zeta)$ for $\zeta \in \Gamma (f,\varepsilon)\cap \pi_D^{-1}(\partial D)$.Let $\eta>0$ be such that $\Phi(\zeta)\stackrel{0.1}{=}\pi_D(\zeta)$ provided $||\pi_D(\zeta)||>1-\eta$. Let $h:D\to D$ denote  the homeomorphism defined by  the formula $h(y)=y||y||^{\log_{0.25}(1-\eta)}$ ( $h$ maps the ball $B({\bf 0},1-\eta)$ onto the ball $B({\bf 0},0.25)$). Let $\tilde\Phi$ be an extension of $\Phi$ over $X\times D$. We define $g:X\times D\to X^\delta \times D$ by $g(x,y)=(f(x,h^{-1}(y)),-\tilde\Phi(x, h^{-1}(y)))$. The mapping $g$  moves points more then $\min(0.5,\varepsilon)$ and this violates pfpp of $X\times D$ because $\delta$ can be arbitrarily small.\end{proof}
 \begin{tw}\label{iloczyn} Let $X\times \mathbb I$ has pfpp and let $Y$ be of surjective semispan 0. Then $X\times Y$ has pfpp as well.\end{tw}
 \begin{proof} Suppose on the contrary that there exists a number $\varepsilon>0$ such that for every number $\delta >0$ there exists  a mapping $f:X\times Y\to X^\delta\times Y^\delta$ moving points more  than $2\varepsilon$. Take $\delta$ such that conditions *) from Lemma \ref{span} and **)from Lemma \ref{essen} with $D=\mathbb I$ are fulfilled. We can  extend $f$ to a mapping $\tilde f:X\times Y^\eta \to X^\delta\times Y^\delta$ moving all points more than $2\varepsilon$. Let $J\subset Y^\eta$ be an arc with Hausdorff distance from $Y$ less than $\delta$.Let $g:X\times J\to X^\delta$ denote the mapping $\pi_{X^\delta}\circ\tilde f$ restricted to $X\times J$. From **) there exists a continuum $K$ in the set  $\Gamma(g,\varepsilon)$ the image of which under $\pi_{Y^\eta}$ is $J$. Thus from *) there exists a point $(x,y)\in K$ such that $\pi_{Y^\eta}\circ \tilde f(x,y)\stackrel{2\varepsilon}{=}(x,y)$ and hence $\tilde f(x,y)\stackrel{2\varepsilon}{=}(x,y)$.\end{proof}

 \begin{tw} For every  family $Y_i,i=1,2\dots,n $ of continua of surjective semispan 0 their product has pfpp. (cf.\cite{Mar})\end{tw}
 \begin{proof} We can start with $\mathbb I^n$ which has pfpp as an AR and successively replace the copies of $\mathbb I$ with $Y_i$ applying Theorem \ref{iloczyn}\end{proof}
 From this, using Proposition \ref{infprod} we get
 \begin{tw}For every family $Y_i, i=1,2,\dots$ of continua with $\sigma(X_i)=0$ their product has pfpp.\end{tw}
 Oversteegen and Tymchatyn  showed in \cite{O-T} that span 0 continua are tree-like,i.e., they are trivial shape one dimensional continua.
 Using this and Propositions \ref{cstrsh} and \ref{jointrsh} we obtain
 \begin{tw} For a span 0 continuum $X$ its  cone $C(X)$ and  suspension $S(X)$  have pfpp. Moreover if $X_i,i=1,\dots n$ are span 0 continua then every join $X_1\star \dots\star X_n$ as well as every multiple application of operations $S,C,\star,\times $ to this spaces has pfpp\end{tw}
 \begin{tw} If $Y$ is an $ANR$ continuum such that $Y\times\mathbb I$ has fixed point property then for every span 0 continuum $X$ the product $Y\times X$ has pfpp.\end{tw}
 \section{Essentially surrounded continua in $\mathbb R^n$}

 The next is a  modification for embeddings in $\mathbb R^n$ of  Professor  Sieklucki's  idea of $q-essential$ mappings.
 \begin{df} A continuum  $X\subset \mathbb R^n$ will be called \emph{essentially surrounded} if for every number  $\varepsilon>0$ there exists an
 $n$-dimensional disc $D$ in $\mathbb R^n$ such that$X\subset \phi  (D)$ and (* the restriction of projection $\pi_D$ of $X\times D$ to the set
 $G=\{(x,y): \rho(x,y)\leq \varepsilon\}$ is essential \end{df}
 It is easy to check that the condition *) in previous definition can be replaced with the following **) there exists a compactum $C$ and mappings
 $\phi:C\to X$ and $\psi C\to D$ such that $\phi(\zeta)\stackrel\leq {\varepsilon}{=}\psi(\zeta))$ and $\psi$ is essential.
 We will consider also some stronger condition
 \begin{df} A continuum  $X\subset \mathbb R^n$ will be called \emph{simply essentially surrounded} if for every number  $\varepsilon>0$ there exists an
 $n$-dimensional disc $D$ in $\mathbb R^n$ such that for every $\delta>0$ there exist a  mapping $\psi: D\to D$ and a mapping $\phi: D\to X^\delta$ such that $\psi$ is homotopic to the identity mapping of $D$ by a homotopy of the pair $(D,\partial D)$ and $\rho(\phi(x),\psi(x)\leq\varepsilon$ for $x\in D$
  \end{df}
 In the fig. is pictured continuum in $\mathbb R^2$ which is simply essentially surrounded (the closure of bounded component of complement of the warsaw circle in $\mathbb R^2$). I do not know any continuum which is essentially surrounded but is not simply essentially surrounded.
 \begin{tw} If $X\subset \mathbb R^n$ is an essentially surrounded continuum then $X$ has pfpp.\end{tw}
 \begin{proof} Consider for an $\varepsilon>0$ a disc $D$ containing $X$ such that there are a compactum $C$ and mappings  $\phi:C\to X$ and $\psi C\to D$ such that $\phi(\zeta)\stackrel\leq {\varepsilon/2}{=}\psi(\zeta))$ and $\psi$ is essential and hence universal. Take $\delta>0$ such that there exists a retraction $r:D^\delta\to D$ shifting points less than $\varepsilon/2$. Now, let $f:X\to D^\delta$. Then $r\circ f:X\to D$, and there exists a point $c\in C$ such that $\psi(c)=r(f(\phi(c))$. But $r(f(\phi(c)\stackrel{\varepsilon/2}{=}f(\phi(c))$ and $\phi(c))\stackrel{\varepsilon/2}{=}\psi(c)$ thus $f(\phi(c))\stackrel{\varepsilon}{=}\phi(c)$.\end{proof}
 \begin{df} We say that a continuum $Y$ has  \emph{close $n$-pith} if for every number $\varepsilon>0$ there exists a number $\delta>0$ such that for every  $\eta>0$ there  exists $p:D^n\to X^\eta$( we will call this mapping \emph{a pith}) such that
    if $\phi:Z\to D^n$ is essential and $\psi:Z\to X^\delta$ is any mapping then there exists a point $z\in Z$ such that
 $p(\phi(z))\stackrel{\varepsilon}{=}\psi(z)$\end{df}
 \begin{pr} If $X\subset \mathbb R^n$ is simply essentially surrounded then $X$ has close $n$-pith.\end{pr}\begin{proof} For a given number $\varepsilon>0$ take a disc $D$ such as in the definition  for $\varepsilon_1=\varepsilon/2$. Let $\delta>0$ be such that there exists a retraction $r:D^\delta\to D$ being an $\varepsilon_1$ shift. For a given $\eta>0$ consider a mapping $\phi:D\to X^\eta$ for which there is $\psi:D\to D$ homotopic to the  identity mapping. Now $\phi$ is an $\eta$-close to $X$ pith. Indeed, if $f:Z\to D$ is essential then the composition $\phi\circ f$ is essential by the Borsuk  Extension ofHomotopy Theorem, hence it is universal. Now, if $g:Z\to X^\delta\subset D^\delta$ then $r\circ g:Z\to D$ and from universality of $\phi\circ f$  there is $z\in Z$ such that $r(g(z))=\phi(f(z))\stackrel{\varepsilon _1}{=}\psi(f(z))$ and
  taking into account that $r(g(z))\stackrel{\varepsilon _1}{=}g(z)$ we have $g(z)\stackrel{\varepsilon}{=}\phi(f(z))$\end{proof}
  \begin{figure} \begin{centering}
\includegraphics[angle=0, width=3in ]{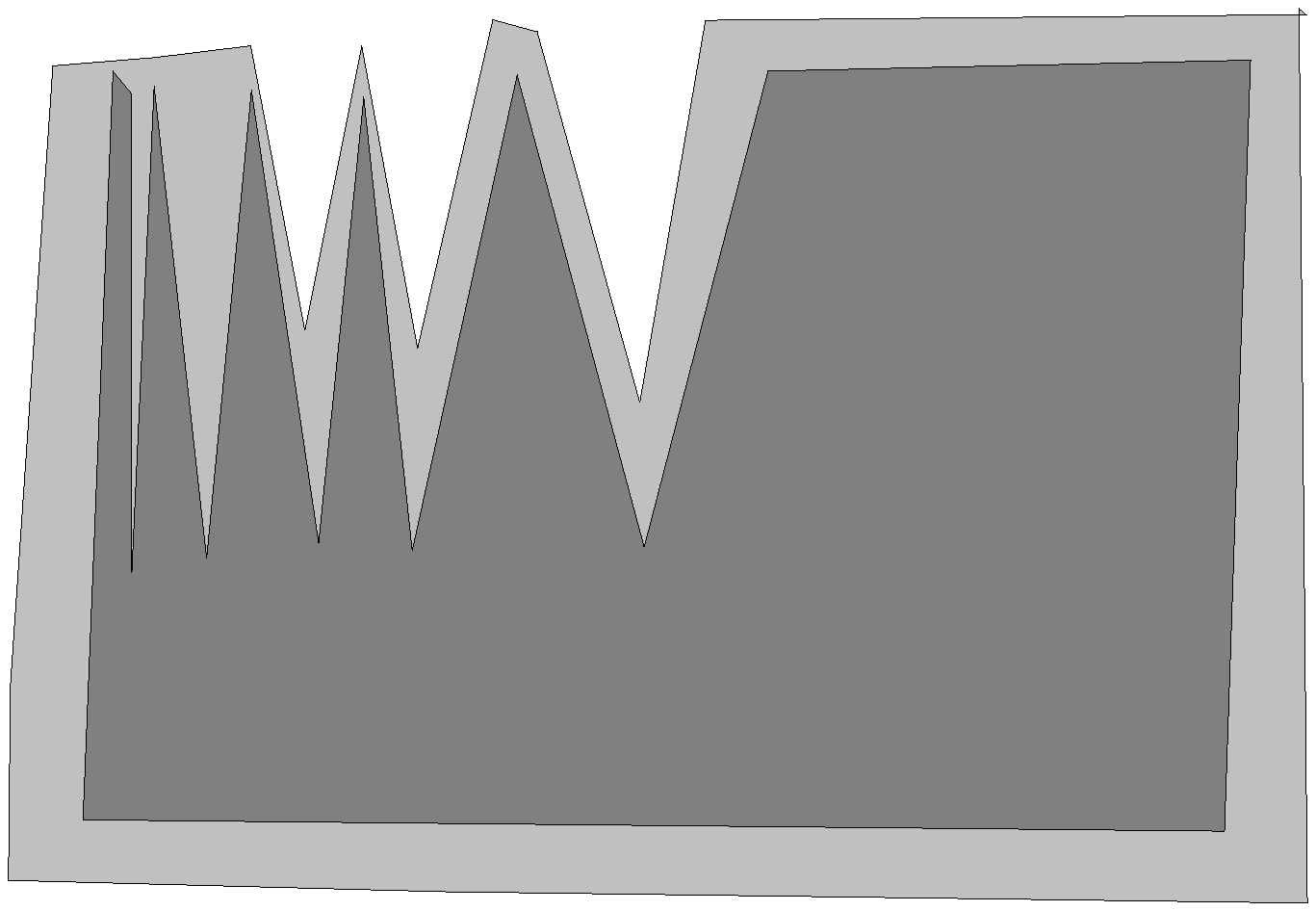}
\caption{Simply surrounded continuum}\label{surr}\end{centering}
\end{figure}
  
\begin{tw} Let $X$ be a continuum such that $X\times D^n$ has pfpp, and let $Y$ has close $n$-pith. Then $X\times Y$ has pfpp.\end{tw}
\begin{proof} The argument goes by replicating the  proof of Theorem \ref{iloczyn}\end{proof}
As an immediate application we get
\begin{st}If a continuum  $Y \subset \mathbb R^n$ is simply essentially surrounded and $X\times D^n$ has pfpp then $X\times Y$ has pfpp too.\end{st}

\end{document}